\documentclass[a4paper,12pt]{article}
\usepackage{amssymb,amsmath,amsthm,latexsym}
\usepackage{amsfonts}
	\usepackage{amsfonts}
\usepackage{graphicx}
\usepackage[pdftex,bookmarks,colorlinks=false]{hyperref}
\usepackage{verbatim}
\usepackage[font=small,labelfont=bf]{caption}
\usepackage{subcaption}
\usepackage[hmargin=1.15in,vmargin=1.15in]{geometry}

\usepackage{fancyhdr}
\newtheorem{theorem}{Theorem}[section]

\newtheorem{corollary}[theorem] {Corollary}
\newtheorem{definition}[theorem]{Definition}

\newtheorem{lemma} [theorem]{Lemma}

\newtheorem{proposition}[theorem]{Proposition}
\newtheorem{remark}[theorem]{Remark}

\title{\bf A Study on the Nourishing Number of Graphs and Graph Powers}
\author{{\bf N K Sudev \footnote{Department of Mathematics, Vidya Academy of Science \& Technology, Thalakkottukara, Thrissur - 680501, email: {\em sudevnk@gmail.com}}} and {\bf K A Germina\footnote{Department of Mathematics, School of Mathematical \& Physical Sciences, Central University of Kerala, Kasaragod, email:{\em srgerminaka@gmail.com}}}}
\date{}
\begin{document}
\maketitle

\begin{abstract}
An integer additive set-indexer is defined as an injective function $f:V(G)\rightarrow 2^{\mathbb{N}_0}$ such that the induced function $g_f:E(G) \rightarrow 2^{\mathbb{N}_0}$ defined by $g_f (uv) = f(u)+ f(v)$ is also injective, where $f(u)+f(v)$ is the sumset of $f(u)$ and $f(v)$. If $g_f(uv)=k~\forall~uv\in E(G)$, then $f$ is said to be a $k$-uniform integer additive set-indexer. An integer additive set-indexer $f$ is said to be a strong integer additive set-indexer if $|g_f(uv)|=|f(u)|.|f(v)|~\forall ~ uv\in E(G)$.  In this paper, we study the characteristics of certain graph classes and graph powers that admit strong integer additive set-indexers.
\end{abstract}


\section{Preliminaries}

For all  terms and definitions, not defined specifically in this paper, we refer to \cite{FH}. For more about different graph classes, we further refer to \cite{BLS}, \cite{JAG} and \cite{GCO}. Unless mentioned otherwise, all graphs considered here are simple, finite and have no isolated vertices.

Let $\mathbb{N}_0$ denote the set of all non-negative integers. For all $A, B \subseteq \mathbb{N}_0$, the sum of these sets is denoted by  $A+B$ and is defined by $A + B = \{a+b: a \in A, b \in B\}$. The set $A+B$ is called the {\em sumset} of the sets $A$ and $B$. 

If either $A$ or $B$ is countably infinite, then their sumset is also countably infinite. Hence, the sets we consider here are all finite sets of non-negative integers. The cardinality of a set $A$ is denoted by $|A|$. 

Based on the concepts of sumsets of finite non-empty sets, an integer additive set-indexer of a given graph $G$ is defined as follows.

\begin{definition}\label{D2}{\rm
\cite{GA} An {\em integer additive set-indexer} (IASI, in short) is defined as an injective function $f:V(G)\rightarrow 2^{\mathbb{N}_0}$ such that the induced function $g_f:E(G) \rightarrow 2^{\mathbb{N}_0}$ defined by $g_f (uv) = f(u)+ f(v)$ is also injective}.
\end{definition}

\begin{definition}{\rm
\cite{GS2} If a graph $G$ has a set-indexer $f$ such that $|g_f(uv)|=|f(u)+f(v)|=|f(u)|.|f(v)|$ for all vertices $u$ and $v$ of $G$, then $f$ is said to be a {\em strong IASI} of $G$.} 
\end{definition}

We use the notation $A<B$ in the sense that $A\cap B=\emptyset$. We notice that the relation $<$ is symmetric, but not reflexive and need not be transitive. By the sequence $A_1<A_2<A_3<\ldots <A_n$, we mean that the given sets are pairwise disjoint. The {\em difference set} of a set $A$, denoted by $D_A$ is the set of all differences between any two elements of $A$.

\begin{lemma}\label{L-RDS}
\cite{GS2} Let $A$, $B$ be two non-empty subsets of $\mathbb{N}_0$. Then, $|A+B|=|A|.|B|$ if and only if their difference sets, denoted by $D_A$ and $D_B$ respectively, follow the relation $D_A<D_B$ (That is, $D_A$ and $D_B$ are pairwise disjoint).
\end{lemma}

\begin{theorem}\label{TSK1}
\cite{GS2} Let each vertex $v_i$ of the  complete graph $K_n$ be labeled by the set $A_i\in 2^{\mathbb{N}_0}$. Then $K_n$ admits a strong IASI if and only if there exists a finite sequence of sets $D_1<D_2<D_3<\ldots,<D_n$ where each $D_i$ is the set of all differences between any two elements of the set $A_i$.
\end{theorem}

\begin{theorem}\label{TSK2}
\cite{GS2} A connected graph $G$ (on $n$ vertices) admits a strong IASI if and only if each vertex $v_i$ of $G$ is labeled by a set $A_i$ in $2^{\mathbb{N}_0}$ and there exists a finite sequence of sets $D_1<D_2<D_3< \ldots <D_m$, where $m\le n$ is a positive integer and each $D_i$ is the set of all differences between any two elements of the set $A_i$.
\end{theorem}

The set $D_i$ of all differences between two elements of the set $A_i$ is called the {\em difference set} of $A_i$ and the relation $<$ is called the {\em difference relation} on $G$. The finite sequence $D_1<D_2<D_3< \ldots <D_m$ of difference sets of the set-labels of vertices of a given graph $G$ is called {\em difference chain} of $G$.

\begin{theorem}\label{TSS}
\cite{GS2} If a graph $G$ admits a strong IASI, then any subgraph $G_1$ of $G$ also admits strong IASI, which is an IASI induced on  $G_1$  by $G$.
\end{theorem}

\begin{definition}{\rm
\cite{GS9} The {\em nourishing number} of a set-labeled graph is the minimum length of the maximal chain of difference sets in $G$. The nourishing number of a graph $G$ is denoted by $\varkappa(G)$.}
\end{definition}

\begin{theorem}\label{T-NNKn}
\cite{GS9} The difference sets of the set-labels of all vertices of a complete graph $K_n$ are pairwise disjoint. That is, the nourishing number of a complete graph $K_n$ is $n$.
\end{theorem}

\begin{remark}
Let $G_1$ be a clique in a strong IASI graph $G$. Then by Theorem \ref{TSS}, $G_1$ also admits the induced strong IASI. Then, the set-labels of all vertices of $G_1$ have pairwise disjoint difference sets. Therefore, the nourishing number of a graph $G$ is the order of a maximal clique in it.
\end{remark}

\begin{theorem}\label{T-NNPC}
\cite{GS9}, The nourishing number of triangle-free graphs is $2$. 
\end{theorem} 

\begin{remark}\label{R-NNBG}
In view of Theorem \ref{T-NNPC}, the nourishing number of trees, cycles and bipartite graphs is $2$.
\end{remark}

A characterisation of strong IASI graphs was done in \cite{GS2} and \cite{GS9}. Based on the results given above, in this paper, we discuss about the nourishing number of certain graph classes and graph powers.


\section{ New Results }

We consider the $r$-th power of a graph only if $r$ is a positive integer. Now, recall the definition of graph powers.

\begin{definition}{\rm 
\cite{BM1} The $r$-th power of a simple graph $G$ is the graph $G^r$ whose vertex set is $V$, two distinct vertices being adjacent in $G^r$ if and only if their distance in $G$ is at most $r$. The graph $ G^2 $ is referred to as the {\em square} of $ G $, the graph $ G^3 $ as the {\em cube} of G.}
\end{definition}

\noindent The following is an important theorem on graph powers.

\begin{theorem}\label{T-Gdiam}
\cite{EWW} If $d$ is the diameter of a graph $G$, then $G^d$ is a complete graph.
\end{theorem}

First note that any power of a complete graph $K_n$ is $K_n$ itself and hence by Theorem \ref{T-NNKn}, the nourishing number of any power of $K_n$ is $n$. Hence, we begin with the study about the nourishing number of a finite powers of complete bipartite graphs.

\begin{proposition}
The nourishing number of the $r$-th power of a complete bipartite graph is 
 \begin{equation*}
 \varkappa(K_{m,n}^r)=
 	\begin{cases}
 	2 & \text{if}~~ r=1\\
 	m+n & \text{if}~~ r\ge 2.
 	\end{cases}
 \end{equation*}
\end{proposition}
\begin{proof}
By Remark \ref{R-NNBG}, the nourishing number of a complete bipartite graph $G$ is $2$. The diameter of a complete graph $G$ is $2$. Therefore, by Theorem \ref{T-Gdiam}, $G^2$ is a complete graph on $m+n$ vertices. Hence, the difference sets of the set-labels of all vertices of $G^r$ must be pairwise disjoint. Therefore, $\varkappa(K_{m,n}^r)=m+n$, for $r\ge 2$.
\end{proof}

Now, we study about the nourishing number of paths. By Theorem \ref{T-NNPC}, the nourishing number of paths is $2$. Hence, we need to study about the nourishing number of arbitrary powers of paths. The following result determines the nourishing number of the $r$-th power of a path.

\begin{theorem}\label{NNPm^r}
The nourishing number of the $r$-th power of a path $P_m$ is 
\begin{equation*}
\varkappa(P_m)=
	\begin{cases}
	r+1 & \text{if}~~ r<m\\
	m+1  & \text{if}~~ r\ge m.
	\end{cases}
\end{equation*}
\end{theorem}
\begin{proof}
Let $P_m=v_1v_2v_3,\ldots v_n;~ n=m+1$. The diameter of $P_m$ is $m$. Then by Theorem \ref{T-Gdiam}, $P_m^m$ is a complete graph on $m+1$ vertices. Therefore, by Theorem \ref{T-NNKn}, for $r\ge m$, the nourishing number of the $r$-th power of $P_m$ is $m+1$.

If $r<m$, then for $1\le i\le n$, the vertices in the set $V'=\{v_i,v_{i+1},v_{i+2}, \ldots, v_{i+r}\}$ are pairwise adjacent and the hence the induced subgraph $\langle V' \rangle$ of $P_m^r$ is a complete graph on $r+1$ vertices. More over, since the vertices $v_i$ and $v_{i+r+1}$ are not adjacent in $P_m^r$, there is no complete subgraph of order $r+1$ or higher. Hence the difference sets of the set-labels of all vertices in $V'$ are pairwise disjoint. Since $i$ is arbitrary, we have the size of a maximal complete subgraph (clique) in $P_m^r$ is $r+1$.  Hence, by Theorem \ref{T-NNKn}, the nourishing number of $P_m^r$ is $r+1$.
\end{proof}

We now proceed to discuss about the admissibility of strong IASI by finite powers of cycles and hence their nourishing number. By Theorem \ref{T-NNPC}, the nourishing numbers of cycles are also $2$. Hence, we need to study about the nourishing number of certain powers of cycles. The following result discusses the nourishing number of the $r$-th power of a cycle.

\begin{theorem}\label{NNCn^r}
The nourishing number of the $r$-th power of a cycle $C_n$ is 
\begin{equation*}
\varkappa(C_n)=
	\begin{cases}
	r+1 & \text{if}~~ r<\lfloor \frac{n}{2} \rfloor\\
	n  & \text{if}~~ r\ge \lfloor \frac{n}{2} \rfloor.
	\end{cases}
\end{equation*}
\end{theorem}
\begin{proof}
Let $P_m=v_1v_2v_3,\ldots v_nv_1$. The diameter of $C_n$ is $\lfloor \frac{n}{2} \rfloor$. Then by Theorem \ref{T-Gdiam}, $C_n^{\lfloor \frac{n}{2} \rfloor}$ is a complete graph on $n$ vertices. Therefore, by Theorem \ref{T-NNKn}, for $r\ge \lfloor \frac{n}{2} \rfloor$, the nourishing number of the $r$-th power of $C_n$ is $n$.

If $r<\lfloor \frac{n}{2} \rfloor$, then for $1\le i\le n$, the vertices in the set $V'=\{v_i,v_{i+1},v_{i+2}, \ldots, v_{i+r}\}$ are pairwise adjacent in $C_n^r$ and the hence the induced subgraph $\langle V' \rangle$ of $C_n^r$ is a complete graph on $r+1$ vertices. Also, since the vertices $v_i$ and $v_{i+r+1}$ are not adjacent in $C_n^r$, there is no complete subgraph of higher order induced by the set of vertices $\{v_i,v_{i+1},v_{i+2} \ldots, v_{i+r+1}\}$. Since the choice of $i$ is arbitrary, we have the size of a maximal complete subgraph (clique) in $P-m^r$ is $r+1$. Hence, by Theorem \ref{T-NNKn}, the difference sets of the set-labels of all vertices in $V'$ are pairwise disjoint and hence the nourishing number of $C_n^r$ is $r+1$.
\end{proof}

A graph that arises interest in this context is a wheel graph which is defined as follows.

\begin{definition}{\rm 
\cite{FH} A {\em wheel graph}, denoted by $W_{n+1}$, is defined as  $W_{n+1}=C_n+K_1$.}
\end{definition}

The nourishing number of the powers of wheel graphs is discussed in the following proposition.

\begin{proposition}
The nourishing number of the $r$-th power of a wheel graph is 
\begin{equation*}
\varkappa(W_{n+1}^r)=
\begin{cases}
3 & \text{if}~~ r=1\\
n+1 & \text{if}~~ r \ge 2.
\end{cases}
\end{equation*}
\end{proposition}
\begin{proof}
When $r=1$, we get the nourishing number of $W_{n+1}$ itself. Let $K_1=v$ and $C_n=v_1v_2v_3\ldots v_nv_1$. For $1\le i\le n$, three vertices $v, v_i$ and $v_{i+1}$ form a triangle in $W_{n+1}$. Since $v_i$ and $v_{i+2}$ are not adjacent in $W_{n+1}$, this triangle is the maximal clique in $W_{n+1}$, where the set-labels of these three vertices have disjoint difference sets. Therefore, $\varkappa(W_{n+1})=3$.

Since any two vertices of a wheel graph $W_{n+1}$ are at a distance $2$ from each other, the square of it is a complete graph on $n+1$ vertices. Hence, by Theorem \ref{T-NNKn}, the difference sets of the set-labels of all vertices of $G^2$ are pairwise disjoint. Therefore, $\varkappa(W_{n+1}^r)=n+1$ for all $r\ge 2$.
\end{proof}

Another interesting graph related to a wheel graph is a {\em helm graph} which is defined as follows.

\begin{definition}{\rm 
\cite{JAG} A {\em helm graph}, denoted by $H_n$, is a graph obtained by adjoining a pendant edge to each of the vertices of the outer cycle of a wheel graph $W_{n+1}$. A helm graph has $2n+1$ vertices and $3n$ edges.}
\end{definition}
 
 \noindent The following theorem discusses the nourishing number of a helm graph and its powers.

\begin{theorem}\label{T-NNHelmG}
The nourishing number of the $r$-th power of a helm graph $H_n$ is given by
\begin{equation*}
\varkappa{H_n^r}=
\begin{cases}
3 & \textrm{if}~~~ r=1\\
n+1 & \textrm{if}~~~ r=2\\
n+4 & \textrm{if}~~~ r=3\\
2n+1 & \textrm{if}~~~ r\ge 4.
\end{cases}
\end{equation*}
\end{theorem}
\begin{proof}
Let $u$ be the central vertex, $V=\{v_1v_2v_3,\ldots, v_n\}$ be the set of vertices of the cycle $C_n$ and $W=\{w_1,w_2,w_3, \ldots,\\ w_n\}$ be the set of pendant vertices in $H_n$. In $H_n$, the central vertex $u$ is adjacent to each vertex $v_i$ of $V$ and each $v_i$ is adjacent to a vertex $w_i$ in $W$.

The maximal clique in $H_n$ is a triangle formed by the vertices $vv_iv_{i+1}v$, for $1\le i \le n$ (in the sense that $v_{n+1}=v_1$), the difference sets of whose set-labels are pairwise disjoint. Therefore, $\varkappa(H_n)=3$.

Since each vertex $v_i$ in $V$ is adjacent to $u$, the distance between any two vertices in $H_n$ is $2$. Hence, the subgraph of $H_n^2$ induced by $V\cup \{u\}$ is a complete graph on $n+1$ vertices in $H_n^2$. Since the distances between any two $w_i$'s is greater than $2$, there is no complete graph of higher order in $H_n^2$. Therefore, $\varkappa(H_n^2)=n+1$.

Since each vertex $w_i$ in $W$ is at a distance at most $3$ from $u$ as well as from all vertices of $V$, for $1\le i \le n$, the subgraph of $H_n^3$ induced by $V\cup \{u, w_{i-1},w_i, w_{i+1}\}$ is a complete graph. Since the distances between between all other vertices in $W$ and $W_i$ are greater than $3$, there is no complete graph of higher order. Therefore, $\varkappa(H_n^2)=n+4$.

Since $u$ is adjacent all $v_i\in V$, the distance between any two vertices in $H_n$ is $4$. Therefore, $H_n^r; ~~ r\ge 4$ is a complete graph on $2n+1$ vertices. Hence by Theorem \ref{T-NNKn}, for $r\ge 4$, the nourishing number of $H_n^r$ is $2n+1$.
\end{proof}

Another graph we consider in this context is a friendship graph whose definition is give below. 

\begin{definition}{\rm
\cite{JAG, GCO} A {\em friendship graph} $F_n$ is the graph obtained by joining $n$ copies of the cycle  $C_3$ with a common vertex. It has $2n+1$ vertices and $3n$ edges.}
\end{definition}

 The following proposition is about the nourishing number of a friendship graph $F_n$.

\begin{proposition}
The nourishing number of the $r$-th power of a friendship graph $F_n$ is 
\begin{equation*}
\varkappa(F_n^r)=
\begin{cases}
3 & \text{if}~~ r=1\\
2n+1 & \text{if}~~ r\ge 2
\end{cases}
\end{equation*}
\end{proposition}
\begin{proof}
The maximal clique in a friendship graph $F_n$ is a triangle. Then, by Theorem \ref{T-NNKn}, $\varkappa(F_n)=3$. Since all the vertices in $F_n$ at a distance $1$ or $2$ among themselves and hence by Theorem \ref{T-Gdiam}, $F_n^2$ is a complete graph on $2n+1$ vertices. Hence, by Theorem \ref{T-NNKn}, $\varkappa(F_n^r)=2n+1$ for all $r\ge 2$.
\end{proof}

Another similar graph structure is a fan graph which is defined as follows.

\begin{definition}
\cite{JAG,BLS} A {\em fan graph}, denoted by $F_{m,n}$,  is defined as the graph join $\bar{K_m}+P_n$.
\end{definition}
\begin{proposition}
The nourishing number of the $r$-th power of a fan graph $F_{m,n}$ is 
\begin{equation*}
\varkappa(F_{m,n}^r)=
\begin{cases}
3 & \text{if}~~ r=1\\
m+n & \text{if}~~ r\ge 2
\end{cases}
\end{equation*}
\end{proposition}
\begin{proof}
The maximal clique in a fan graph $F_{m,n}$ is a triangle. Then, by Theorem \ref{T-NNKn}, $\varkappa(F_n)=3$. Since all the vertices in $F_{m,n}$ at a distance $1$ or $2$ among themselves and hence by Theorem \ref{T-Gdiam}, $F_{m,n}^2$ is a complete graph on $m+n$ vertices. Hence, by Theorem \ref{T-NNKn}, $\varkappa(F_{m,n}^r)=m+n$ for all $r\ge 2$.
\end{proof}

Now, we proceed to discuss about the nourishing number of split graphs and and their powers. A {\em split graph} is defined as follows.

\begin{definition}{\rm
\cite{BLS} A {\em split graph}, denoted by $G(K_r,S)$, is a graph in which the vertices can be partitioned into a clique $K_r$ and an independent set $S$. A split graph is said to be a {\em complete split graph} if every vertex of the independent set $S$ is adjacent to every vertex of the the clique $K_r$ and is denoted by $K_S(r,s)$, where $r$ and $s$ are the orders of $K_r$ and $S$ respectively. }
\end{definition}

The following theorem estimates the nourishing number of a split graph.

\begin{theorem}\label{T-NNSpG}
The nourishing number of the $r$-th power of a split graph $G=G(K_r,S)^r$ is given by
\begin{equation*}
\varkappa(G^r)=
\begin{cases}
r & \text{if no vertex of $S$ is adjacent to all vertices of $K_r$}\\
r+1 & \text{if some vertices of $S$ are adjacent to all vertices of $K_r$}\\
r+l & \text{if $r=2$}\\
r+s & \text{if $r\ge 3$};
\end{cases}
\end{equation*}
where $l$ is the maximum number of vertices of $S$ which are adjacent to the same vertex of $K_r$ and $s$ is the number of non-isolated vertices in $S$.
\end{theorem}
\begin{proof}
Let $G$ be a split graph consisting of a clique $K_r$ and an independent set $S$ that contains $s$ vertices. The isolated vertices in $S$ do not have any impact on the adjacency in the ppowers of $G$ and hence we need not consider isolated vertices in $S$. Hence, without loss of generality, let $S$ has no isolated vertices in $G$. 

If there is no vertex of $S$ which has adjacency with every vertex of $K_r$, then $K_r$ itself is the maximal clique in $G$ and hence $\varkappa(G)=r$.

Assume that some vertices of $S$ are adjacent to all vertices of $K_r$. Let $w_i$ be a vertex in $S$ that is adjacent to every vertex of $K_r$. Then $K_r+\{w_i\}$ is a complete subgraph of $G$.Since no vertices in $S$ are adjacent among themselves, This subgraph is a maximal clique in $G$. Therefore, $\varkappa(G)=r+1$.

Let $S_1=\{w_1,w_2,\ldots, w_l\}$ be the set of vertices in $S$ that are adjacent to a vertex, say $v_i$ in $K_r$. Then each of them is at distance $2$ among themselves and from all other vertices of $K_r$. Therefore, $G_1=K_r\cup S_1$ is a complete graph in $G^2$. Since, $S_1$ is the maximal set of that kind, $G_1$ is a maximal clique in $G^2$. Therefore, $\varkappa(G^2)=r+l$.

All vertices in a split graph $G$ are at a distance at most $3$ in $G$. Therefore, $G^3$ is a complete graph on $r+s$ vertices. Therefore, $\varkappa(G^r)=r+s$, for all values of $r\ge 3$. 
\end{proof}

In view of Theorem \ref{T-NNSpG}, we establish the following result on complete split graphs.

\begin{corollary}
The nourishing number of the $k$-th power of a complete split graph $G=K_S(r,s)$ is given by
\begin{equation*}
\varkappa(G^r)=
\begin{cases}
r+1 & \text{if $k=1$}\\
r+s & \text{if $k\ge 2$};
\end{cases}
\end{equation*}
\end{corollary}
\begin{proof}
Since every vertex of $S$ is adjacent to all vertices of $K_r$ in $G$, by Theorem \ref{T-NNSpG}, $\varkappa(G)=r+1$. Also, in $G$, being adjacent to all vertices of $K_r$, all vertices in $S$ are at a distance $2$ among themselves, $K_r\cup S$ is a complete graph on $r+s$ vertices in $G^2$. Therefore, $\varkappa(G^k)=r+s$ for all $k\ge 2$.
\end{proof}

\begin{definition}{\rm
\cite{BLS,GCO} An {\em $n$-sun} or a {\em trampoline}, denoted by $S_n$,  is a chordal graph on $2n$ vertices, where $n\ge 3$, whose vertex set can be partitioned into two sets $U = \{u_1,u_2,u_3,\ldots, u_n\}$ and $W = \{w_1,w_2,w_3,\ldots, w_n\}$ such that $W$ is an independent set of $G$ and $w_j$ is adjacent to $u_i$ if and only if $j=i$ or $j=i+1~(mod ~ n)$. {\em A complete sun} is a sun $G$ where the induced subgraph $\langle U \rangle$ is complete.} 
\end{definition}

The sun graphs with $\langle U \rangle$ is a cycle is one of the most interesting structures among the sun graphs.  The following theorem determines the nourishing number of a sun graph $G$ whose non-independent set of vertices induces a cycle in the graph $G$.

\begin{theorem}\label{T-NNSunG}
If $G$ is an $n$-sun graph with $\langle U \rangle = C_n$, then the nourishing number of $G^r$ is given by
\begin{equation*}
\varkappa(G^r)=
\begin{cases}
2r+1 & \text{if}~~~ r<\lfloor \frac{n}{2}\rfloor\\
2(n-1) & \text{if}~~~ r=\lfloor \frac{n}{2}\rfloor ~\text{and if~ $n$ ~ is odd}\\
2n-1 & \text{if}~~~ r=\lfloor \frac{n}{2}\rfloor ~\text{and if~ $n$ ~ is even}\\
2n & \text{if}~~~ r\ge \lfloor \frac{n}{2}\rfloor +1;
\end{cases}
\end{equation*}
\end{theorem}
\begin{proof}
Let $r< \lfloor \frac{n}{2}\rfloor$, then the subgraph of $G^r$ induced by the vertex set $V'=\{w_{n-r+1}, w_{n-r+2},w_{n-r+3},\ldots,w_n, v_{n-r+1},v_{n-r+2},v_{n-r+3}, \ldots, v_n,v_1\}$ is a clique of order $2r+1$ in $G^r$. Since neither $w_{n-r}$ is adjacent to $v_{i+1}$ nor $w_{n-r+1}$ is adjacent to $v_{i+2}$, this clique is a maximal clique in $G^r$. Therefore, by Theorem \ref{T-NNKn}, $\varkappa(G^r)=2r+1$.

If $r= \lfloor \frac{n}{2} \rfloor$ and $n$ is odd, then the vertex $w_n$ is adjacent to every vertex of $V$ and to the vertices of $W$ except $w_r$ and $w_{r+1}$. Then, $V'=V\cup[W-\{w_r,w_{r+1}\}]$ induces a maximal clique of order $2n-2$ in $G$. Therefore, by Theorem \ref{T-NNKn}, if $r= \lfloor \frac{n}{2} \rfloor$ and $n$ is odd, then the nourishing number of $G^r$ is $2(n-1)$.

If $r= \lfloor \frac{n}{2} \rfloor +1$ and $n$ is even, then the vertex $w_n$ is adjacent to every vertex of $V$ and to the vertices of $W$ except $w_r$ in $G^r$. Then, $V'=V\cup[W-\{w_r\}]$ induces a maximal clique of order $2n-1$ in $G^r$. Therefore by Theorem \ref{T-NNKn}, if $r= \lfloor \frac{n}{2} \rfloor$ and $n$ is odd, then the nourishing number of $G^r$ is $2n-1$.

Since all the vertices in $S_n$ are at a distance at most $1+\lfloor \frac{n}{2} \rfloor$, $G^r$ is a complete graph for all $r\ge 1+\lfloor \frac{n}{2} \rfloor$ on $2n$ vertices. Hence by \ref{T-NNKn}, for $r\ge \lfloor \frac{n}{2} \rfloor+1$, the nourishing number of $G^r$ is $2n$.
\end{proof}

\begin{theorem}\label{T-NNKSunG}
The nourishing number of a complete $n$-sun graph $G$ is given by
\begin{equation*}
\varkappa(G^r)=
\begin{cases}
n & \text{if}~~~ r=1 ~\text{and if}~~~ \langle U \rangle \text{is triangle-free}\\
n+1 & \text{if}~~~ r=2\\
2n & \text{if}~~~ r\ge 3;
\end{cases}
\end{equation*}
\end{theorem}

\begin{proof}
If $r=1$, $K_n$ itself is the largest clique in the the complete sun graph $G$. Therefore, $\varkappa(G)=n$.

Each vertex $w_i$ in the independent set $W$ is adjacent to two vertices of $U$ and is at a distance $2$ from all other vertices of $U$. Hence, each vertex $w_i$ is adjacent to all vertices of $U$ in $G^r$. The subgraph of $G^2$ induced by $U\cup \{w_i\}$ is a clique in $G^2$. Since no two vertices in $W$ are at a distance $2$, this clique is a maximal clique in $G^2$. Therefore, $\varkappa(G^2)=n+1$.

Since all vertices in $G$ are at a distance at most $3$, the cube and higher powers of $G$ are complete graphs on $2n$ vertices. Therefore, by \ref{T-NNKn}, $\varkappa(G^r)=2n$ for $r\ge 3$.
\end{proof}

Another important graph is a sunlet graph. An $n$-sunlet graph is defined as follows.
\begin{definition}{\rm
\cite{GCO} The {\em $n$-sunlet} graph is the graph on $2n$ vertices obtained by attaching one pendant edge to each vertex of a cycle $C_n$.}
\end{definition}

The following result discusses about the nourishing number of an $n$-sunlet graph. 

\begin{theorem}
The nourishing number of the $r$-th power of an $n$-sunlet graph $G$ is
\begin{equation*}
\varkappa(G^r)=
	\begin{cases}
	2r & \text{if}~~ r<\lfloor \frac{n}{2} \rfloor +1\\
	2(n-1) & \text{if}~~ r= \lfloor \frac{n}{2} \rfloor+1 ~\text{and ~$n$~ is odd}\\ 
	2n-1 & \text{if}~~ r= \lfloor \frac{n}{2} \rfloor+1 ~\text{and ~$n$~ is even}\\
	2n & \text{if}~~ r\ge \lfloor \frac{n}{2} \rfloor+2.
	\end{cases}
\end{equation*} 
\end{theorem}
\begin{proof}
The proof of this proposition is similar to the proof of Theorem \ref{T-NNSunG}. Let $V=\{v_1,v_2,v_3,\ldots, v_n\}$ be the vertex set of the cycle $C_n$ and $W=\{w_1,w_2,w_3, \ldots w_n\}$ be the set of pendant vertices in $G$.

If $r< \lfloor \frac{n}{2} \rfloor +1$, the vertex set $V'=\{w_{n-r+2},w_{n-r+3},w_{n-r+4}\ldots,w_n, v_{n-r+1},v_{n-r+2}, \\ v_{n-r+3},\ldots,v_n,v_1\}$ induces a clique in $G$ on $2r$ vertices. Since there is no other vertex outside $V'$ which is at a distance at most $r$ from all vertices in $V'$, the induced graph $\langle V' \rangle$ is a maximal clique in $G^r$. Therefore, for $r\le 1+\lfloor \frac{n}{2} \rfloor$, $\varkappa(G)=2r$.

If $r= \lfloor \frac{n}{2} \rfloor +1$ and $n$ is odd, then the vertex $w_n$ is adjacent to every vertex of $V$ and to the vertices of $W$ except $w_r$ and $w_{r+1}$. Then, $V'=V\cup[W-\{w_r,w_{r+1}\}]$ induces a maximal clique of order $2n-2$ in $G$. Therefore, in this case, the nourishing number of $G^r$ is $2(n-1)$.

If $r= \lfloor \frac{n}{2} \rfloor +1$ and $n$ is even, then the vertex $w_n$ is adjacent to every vertex of $V$ and to the vertices of $W$ except $w_r$ in $G^r$. Then, $V'=V\cup[W-\{w_r\}]$ induces a maximal clique of order $2n-1$ in $G^r$. Therefore, in this case, the nourishing number of $G^r$ is $2n-1$.

Since the maximum distance between the vertices in $W$ is $2+\lfloor \frac{n}{2} \rfloor$, diameter of $G$ is $2+\lfloor \frac{n}{2} \rfloor$. Hence, by Theorem \ref{T-Gdiam}, $G^r$ is a complete graph for all $r\ge 2+\lfloor \frac{n}{2} \rfloor$ on $2n$ vertices. Therefore, the difference sets of the set-labels of all vertices of $G^r$ are pairwise disjoint. Hence, for $r\ge 2+\lfloor \frac{n}{2} \rfloor$ the nourishing number of $G^r$ is $2n$.
\end{proof}


\section{Conclusion}

In this paper, we have established some results on the nourishing number of certain graphs and graph powers. The admissibility of strong IASI by various other graph classes, graph operations and graph products and finding the corresponding nourishing numbers are open. Determinning the nourishing numbers of powers of bipartite graphs, chordal graphs, crown graphs, prism graphs etc. are some of the open problems. 

Besides the above said problems, there are several other graph classes and structures for which the admissibility of strong IASIs are yet to be verified and the corresponding nourishing numbers are to be determined. 

More over, the properties and characteristics of strong IASIs, both uniform and non-uniform, are yet to be investigated. The problems of establishing the necessary and sufficient conditions for various graphs and graph classes to have certain IASIs are also open.





\end{document}